\documentclass[reqno]{amsart}
\UseRawInputEncoding

\usepackage{amsmath,amsfonts,amssymb,amsthm,amscd,latexsym,cite,comment}

\usepackage{mathrsfs}
\usepackage{color}
\usepackage{enumerate}
\usepackage{comment}
\usepackage{a4wide}
\theoremstyle{plain}
\newtheorem{thm}{Theorem}[section]
\newtheorem{theorem}[thm]{Theorem}

\newtheorem{lemma}[thm]{Lemma}

\newtheorem{proposition}[thm]{Proposition}

\theoremstyle{definition}

\theoremstyle{remark}
\newtheorem{remark}[thm]{Remark}

\newcommand{\cL}{{\mathcal L}}

\usepackage{color}
 
\newcommand{\norm}[1]{\left\lVert#1\right\rVert}

\newcommand{\tnorm}[1]{\left\vert\kern-0.25ex\left\vert\kern-0.25ex\left\vert #1 
	\right\vert\kern-0.25ex\right\vert\kern-0.25ex\right\vert}

\usepackage{hyperref}					
\hypersetup{colorlinks,
	linkcolor=blue,%
	citecolor=blue}

\usepackage{float}
\newcommand{\R}{\mathbb R}
\newcommand{\Z}{\mathbb Z}
\newcommand{\N}{\mathbb N}
\newcommand{\G}{\mathbb G} 
\newcommand{\one}{\mathbf 1}
\newcommand{\supp}{\operatorname{supp}}

\newcommand{\dd}{\,\mathrm d}

\title[A counterexample to two compactness conjectures]
{A Counterexample to Two Compactness Conjectures of Brudnyi}

\author[Jinghao Huang]{Jinghao Huang}
\address[Jinghao Huang]{Institute for Advanced Study in Mathematics, Harbin Institute of Technology, Harbin 150001, China \emph{E-mail~:} {\tt jinghao.huang@hit.edu.cn} }

\author[Ran Xu]{Ran Xu}
\address[Ran Xu]{Institute for Advanced Study in Mathematics, Harbin Institute of Technology, Harbin 150001, China \emph{E-mail~:} {\tt xxurann@stu.hit.edu.cn} }

\thanks{\footnotesize
Jinghao Huang was supported by the NNSF of China (No. 12031004, 12301160 and  12471134).
}


\begin{document}

\begin{abstract}
We construct a quasi-Banach function lattice $X$ on $\R$ whose quasi-norm is
order continuous and invariant under every translation.  Nevertheless, $X$
contains a bounded sequence with  common compact support which is equicontinuous in $X$ and has no norm-convergent subsequence. This provides  a
counterexample to Conjectures~2.5 and 2.11 proposed by Brudnyi in \cite{Brudnyi2015}.
\end{abstract}

\maketitle

\section{Introduction}

The Kolmogorov--Riesz compactness theorem characterizes relatively compact
subsets of $L^p(\R^d)$, $1\le p<\infty$, by boundedness, uniform
continuity under translations, and uniform vanishing at infinity  \cite{Kolmogoroff1931, Riesz1933, Frechet1937}. It is
the $L^p$ counterpart of the classical Arzel\`a--Ascoli theorem \cite{Ascoli1884, Arzela1895} and provides a model
for compactness criteria in more general function spaces; see
\cite{HancheOlsenHolden2010,Krotov2012} for historical accounts and further
references.

Brudnyi studied compactness beyond the Banach setting
in  his class
$\cL$.
This class consists of translation-invariant
Fr\'echet lattices  $(X,\norm{\cdot})$  (briefly, $F$-lattices)  of measurable
functions  on
$\R^d$ whose $F$-norms are absolutely continuous  \cite{Brudnyi2015}. 
Absolute continuity means that, for every
$f\in X$, we have  
\begin{equation}\label{eq:ac-convention}    
    \lim_{n\to\infty}\|f\one_{S_n}\|=0,   \quad\text{whenever }\one_{S_n}\to0
    \text{ almost everywhere},
\end{equation}
where  $(S_n)$ is a sequence of measurable sets and $\one_{S_n}$ stands for the indicator of $S_n$.
Brudnyi proved that, for $X\in\cL$, a subset
$\Sigma\subseteq X$ is precompact if and only if it is
bounded, equicontinuous, and
\emph{norm absolutely equicontinuous}
\cite[Theorem~2.3]{Brudnyi2015}, where the last condition requires 
the convergence in
\eqref{eq:ac-convention} to be uniform over $f\in\Sigma$.

He then asked whether norm absolute equicontinuity
could be replaced by the simpler condition of \emph{vanishing at infinity},
 that is,
\begin{equation*}
\lim_{n\to\infty}\sup_{f\in\Sigma}
    \|f\one_{\{|x|\ge n\}}\|=0.
\end{equation*}
The relevant conjectures are stated below.

\medskip
\emph{
If $X\in\cL$ and $\Sigma\subseteq X$ is bounded,
equicontinuous, and vanishes uniformly at infinity, 
then $\Sigma$ is precompact in $X$   \cite[Conjecture 2.5]{Brudnyi2015}.
}

\medskip
\emph{
If $X$ is a quasi-Banach lattice in $\cL$, then a subset of $X$ is
precompact if and only if 
it
is equicontinuous  and vanishes  at infinity  \cite[Conjecture 2.11]{Brudnyi2015}.
}
\medskip

The boundedness condition cannot, in general, be omitted in the
$F$-lattice case; see \cite[Example 2.8]{Brudnyi2015}
for a counterexample. 
Brudnyi established the corresponding compactness criteria for several important classes, including
Orlicz classes, generalized Orlicz classes, and Banach lattices under appropriate assumptions \cite{Brudnyi2015}. 
Related compactness criteria in noncommutative quasi-Banach spaces were obtained in \cite{HuangNessipbayevSukochevZanin2025}.

In the present paper, however, we show that
Conjecture~2.5 and the sufficiency direction of
Conjecture~2.11 in \cite{Brudnyi2015} fail in general.
We construct the following  counterexample on $\R$ in which
all functions are supported in a fixed compact
interval, so uniform vanishing at infinity
holds automatically.

\begin{theorem}\label{thm:main}
There exist a $1/2$-Banach function lattice $X$ on $\R$, a compact interval
$Q$, and a sequence $(f_n)_{n\ge1}\subseteq X$ such that:
\begin{enumerate}
\item The operators $(\tau_h)_{h\in\R}$ form a strongly continuous group
of surjective isometries of $X$.
\item The quasi-norm of $X$ is order continuous.
\item For every $n$,
\[
 \frac1{16}\le\|f_n\|_X\le1,
 \qquad \supp f_n\subseteq Q.
\]
\item The family $(f_n)_{n\ge1}$ is equicontinuous in $X$:
\[
 \lim_{h\to0}\sup_n\|\tau_hf_n-f_n\|_X=0.
\]
\item The sequence $(f_n)_{n\ge1}$ has no norm-convergent subsequence.
\end{enumerate}
Consequently, $X$ 
belongs to
Brudnyi's class $\cL$, and the family $(f_n)_{n\ge1}$
provides a counterexample to Conjectures~2.5 and
~2.11. 
\end{theorem}

The constructed lattice does not have the Fatou property;
see Remark~\ref{rem:fatou}. Consequently, the compactness results in
\cite{CaetanoGogatishviliOpic2016,GuoZhao2020},
whose function-space assumptions include the Fatou
property, do not apply to $X$.
By \cite[Theorem~2.3]{Brudnyi2015}, the family $\{f_n:n\ge1\}$ is not norm absolutely
equicontinuous, see also Remark~\ref{notabsonorm}.

\section{Preliminaries}  Let $L^0(\mathbb{R})$ denote the vector lattice of all real- or complex-valued Lebesgue measurable functions on $\mathbb{R}$ that are finite almost everywhere, where functions equal almost everywhere are identified. 
The order and lattice operations in $L^0(\mathbb{R})$ are understood pointwise almost everywhere. 
For  a Lebesgue measurable subset  $S$ of
$\mathbb{R}$, let  $|S|$ denote its Lebesgue measure.

We begin by recalling the notions of $F$-norms and
$F$-lattices used in \cite[Section~1]{Brudnyi2015}.
Let $\mathbb{K}=\mathbb{R}$ or $\mathbb{C}$, and let $E$
be a vector space over $\mathbb{K}$. A functional
$\varrho\colon E\to[0,\infty)$ is called an \emph{$F$-norm}
if
\begin{enumerate}
    \item $\varrho(f)=0$ if and only if $f=0$;
    \item $\varrho(-f)=\varrho(f)$ for every $f\in E$;
    \item $\varrho(f+g)\leq\varrho(f)+\varrho(g)$
          for all $f,g\in E$;
\end{enumerate}
and scalar multiplication is continuous with respect to
the metric
\[
    d_{\varrho}(f,g):=\varrho(f-g).
\]
If this metric is complete, $(E,\varrho)$ is called an
\emph{$F$-space}.

An $F$-space $(X,\varrho)$ with
$X\subseteq L^0(\mathbb{R})$ is called an
\emph{$F$-lattice} if it has the following lattice property:
whenever $f\in L^0(\mathbb{R})$ and $g\in X$ satisfy
$|f|\leq|g|$ almost everywhere, one has
\[
    f\in X
    \qquad\text{and}\qquad
    \varrho(f)\leq\varrho(g).
\]

An $F$-lattice $(X,\varrho)$ on $\mathbb{R}$ is said
to be \emph{translation-invariant} if, for every
$h\in\mathbb{R}$, the translation operator
$ 
    \tau_h f:=f(\cdot+h)
$ 
maps $X$ into itself and satisfies
$ 
    \varrho(\tau_h f)=\varrho(f), $ $f\in X 
$ \cite[Definition~1.7]{Brudnyi2015}.

The translation group $(\tau_h)_{h\in\mathbb{R}}$
is said to be \emph{strongly continuous} on a translation-invariant $F$-lattice $X$
if \begin{align*} \lim_{h\to0}\varrho_X(\tau_hf-f)=0, \qquad \text{for every }  f\in X. 
\end{align*}
See \cite[Section~1]{ConejeroMullerPeris2007}
for the corresponding notion of strong continuity
for operator semigroups on $F$-spaces.

The $F$-norm of an $F$-lattice $(E,\varrho)$ is said
to be \emph{order continuous} if
$  \varrho(f_n)\to 0
$
for every sequence $(f_n)_{n\geq1}\subseteq E$ satisfying
$0\leq f_n\downarrow0$ almost everywhere. 
Order continuity implies absolute continuity in the sense of \eqref{eq:ac-convention}. Indeed, if \(\mathbf1_{S_n}\to0\) a.e., then for every $f\in E$, 
\[
u_n:=|f|\mathbf1_{\mathop{\bigcup}\limits_{k\ge n}S_k}\downarrow0
\quad\text{a.e.},
\qquad |f\mathbf1_{S_n}|\le u_n \leq |f|, 
\] and so, 
$  \rho(f\one_{S_n}) \le\rho(u_n)\to0.
$

We next introduce $p$-Banach lattices \cite[p.322]{Kalton1980}, which can be viewed
as a special class of $F$-lattices.
Let $0<p\leq1$. A functional
$\norm{\cdot}_E\colon E\to[0,\infty)$ is called a
\emph{$p$-norm} if
\begin{enumerate}
    \item $\|f\|_E=0$ if and only if $f=0$;
    \item $\|\lambda f\|_E=|\lambda|\|f\|_E$
          for all $\lambda\in\mathbb{K}$ and $f\in E$;
    \item
$
        \|f+g\|_E^p\leq\|f\|_E^p+\|g\|_E^p,
$   for all $f,g\in E$.
\end{enumerate}
A vector space equipped with a $p$-norm is called a
\emph{$p$-Banach space} if it is complete for the metric
\[
    d_p(f,g):=\|f-g\|_E^p.
\]
If, in addition, $E\subseteq L^0(\mathbb{R})$ satisfies
the lattice property above with $\varrho$ replaced by  $p$-norm 
$\norm{\cdot}_E$, then $E$ is called a
\emph{$p$-Banach lattice}.

Every $p$-Banach lattice $(E, \norm{\cdot}_E)$ becomes an $F$-lattice under
the functional
\[
    \varrho_E(f):=\|f\|_E^p.
\]
Then, $\varrho_E$ is a $p$-homogeneous $F$-norm satisfying
\[
\varrho_E(f+g)\leq\varrho_E(f)+\varrho_E(g),
   \qquad
   \varrho_E(\lambda f)=|\lambda|^p\varrho_E(f).
\] 
Since
\[
   \varrho_E(f_n-f)\to0
   \quad\Longleftrightarrow\quad
   \|f_n-f\|_E\to0,
\]
the two functionals induce the same topology.  
 Moreover, the  boundedness, translation isometry, compactness, equicontinuity, order-continuity, and uniform
vanishing conditions used below are unaffected by passing between
these two equivalent conventions.

\section{Proof of Theorem \ref{thm:main} }
\subsection{Sparse sets and covering submeasures}
We begin with a combinatorial lemma that provides
the sparse sets needed for our construction.

For an integer $N\ge1$, let $\G_N:=\Z/N\Z$. 
We write $\mathbb P$ for probability, $\mathbb E$ for expectation,
and $\operatorname{Var}$ for variance.

\begin{lemma}\label{lem:sparse}
Let $K\ge2$ be an integer and set $N:=K^{8K}$. Then, there exists a subset 
$A\subseteq \G_N$ such that
\begin{equation}\label{eq:sparse}
|A|\leq\frac{2N}{K} 
 \qquad
 \text{and} \qquad \bigcap_{j=1}^s(A+g_j)\ne\varnothing 
\end{equation}
for  every integer  $1\le s\le2K$ and  every choice of  $ g_1,\ldots,g_s\in\G_N$. 
Here, 
  $|A|$ denotes the cardinality of $A$, 
and  $A+g:=\{a+g:a\in A\}$, with addition taken modulo $N$.
\end{lemma}

\begin{proof}

Choose a random subset $A\subseteq\G_N$ by including
each element of \(\G_N\) independently with probability $1/K$.   The cardinality \(|A|\) follows a binomial distribution with parameters \(N\) and \(1/K\).
In particular, for any distinct \(y_1,\ldots,y_r\in \G_N \), independence gives 
\begin{align}\label{frac{1}{K^r}}
\mathbb P(y_1,\ldots,y_r\in A)=K^{-r}.
\end{align} 
Moreover, \begin{equation}\label{eq:cardinality} \mathbb E|A|=\frac{N}{K}, \qquad   \operatorname{Var}(|A|)
    =\frac NK\left(1-\frac1K\right).
\end{equation}

We first estimate the probability of the event \begin{align}\label{defC}
     C   :  =\left\{|A|>2 \mathbb E|A| \right\} 
    =\left\{|A|>\frac{2N}{K}\right\}
=\left\{|A|-\frac NK>\frac NK\right\}
 \subseteq
\left\{\left||A|-\frac NK\right|>\frac NK\right\}.
\end{align} 
Chebyshev's inequality gives
\begin{align}
    \mathbb P(C)
    \le
    \mathbb P\left(
        \left||A|-\frac NK\right|>\frac NK
    \right)
   \le
    \frac{\operatorname{Var}(|A|)}{(N/K)^2}
    =\frac{K-1}{N}
    <\frac1{K^2}
    \le\frac14,
    \label{eq:second-quarter}
\end{align}
where we used $N=K^{8K}\ge K^3$ and $K\ge2$.

We now turn to the intersection condition, beginning
with a fixed set of translation parameters. 
Fix $H\subseteq \G_N$  with $1\le |H| \le 2K$.
A point $x\in \G_N$ belongs to every translate $A+g$, $g\in H$,
if and only if $x-H\subseteq A$, where
$x-H:=\{x-g:g\in H\}$. 
Since the map $g\mapsto x-g$ is injective,  $x-H$ consists of
$|H| $ distinct elements. 
Thus the event
\begin{equation}
\label{eq:success-event}
 S_x(H):=\left\{x\in\bigcap_{g\in H}(A+g)\right\}
 =
 \big\{
 x-H\subseteq A
  \big\}
\end{equation}
has probability \begin{align}\label{eq:success-probability}
\mathbb P(S_x(H))
\stackrel{\eqref{frac{1}{K^r}}}{=} 
 K^{-|H|}
\overset{\scriptscriptstyle  1\le |H| \le 2K, K \geq 2}{\ge}
 K^{-2K} .
 \end{align}

For this fixed $H$, 
let
\begin{equation}\label{eq:fixed-bad-event}
 B_H:=\left\{\bigcap_{g\in H}(A+g)=\varnothing\right\}.
\end{equation}
To bound $\mathbb P(B_H)$, we select finitely many centers  
whose corresponding events $S_{x}(H)$ are mutually
independent. To ensure this independence, we choose the associated
sets $x-H$ to be pairwise disjoint.

For a fixed center $x \in\G_N$, a center $y\in\G_N$
satisfies $(x-H)\cap(y-H)\ne\varnothing$ if and only if there exist
$g,h\in H$ such that $x-g=y-h$, or equivalently $y=x+h-g$. 
Hence, the set of all \(y\in \G_N\) such that
$
(x-H)\cap(y-H)\neq\varnothing
$
is precisely \(x+(H-H)\), where
$ 
 H-H:=\{h-g:h,g\in H\}.
$ 
Thus, 
\begin{equation*}
 |x+(H-H)|=|H-H|\le |H|^2  \overset{\scriptscriptstyle  1\le |H| \le 2K}{\le}
 (2K)^2.
\end{equation*}

Start with all $N$ elements of $\G_N$ as candidate
centers. 
At each step, select a center $x$ and discard all
candidate centers $y$ satisfying
$(x-H)\cap(y-H)\ne\varnothing$, including $x$ itself.
Repeat until no candidate remains, and denote the
selected centers by $x_1,\ldots,x_M$.
By construction, the sets $x_i-H$ are pairwise disjoint.
Each step discards at most $(2K)^2$ centers, so
\begin{align}\label{eq:number-centers}
    N\le M(2K)^2,
    \qquad\text{hence}\qquad
    M\ge\frac{N}{(2K)^2}.
\end{align}
The centers are chosen independently of $A$.
Since the events $S_{x_i}(H)$ depend on the inclusion
of disjoint sets of points in $A$, they are mutually
independent.

If $B_H$ occurs, none of the selected centers is a common point.
Hence, 
$$B_H\stackrel{\eqref{eq:success-event}, \eqref{eq:fixed-bad-event}}{\subseteq}
\bigcap_{i=1}^{M}S_{x_i}(H)^C,$$ where the superscript $C$
denotes the complementary event. Independence  implies
\begin{align} \label{eq:fixed-failure-bound}
 \mathbb P(B_H)
 &\nonumber \,\,\,\,  \,\,\,\,  \le\mathbb P\left(\bigcap_{i=1}^{M}S_{x_i}(H)^C\right)
=\prod_{i=1}^{M}\mathbb P\bigl(S_{x_i}(H)^C\bigr)
 =\prod_{i=1}^{M}\bigl(1-\mathbb P(S_{x_i}(H))\bigr)\notag \stackrel{\eqref{eq:success-probability}}{=}
 \left(
 1- K^{-|H|}  \right)^M
\\
 &\,\, \,  \,\,\,\,  \stackrel{\eqref{eq:success-probability}}{\leq} 
(1-K^{-2K})^M \overset{\scriptscriptstyle  K \geq 2, \eqref{eq:number-centers} }{\leq}  (1-K^{-2K})^{N/(2K)^2}
\le \exp\left(-\frac{NK^{-2K}}{(2K)^2}\right)\notag
\\
&\,\,\,\, \,\,\,\, =\exp\left(-\frac{N}{4K^{2K+2}}\right), 
\end{align}
where we used the fact that  $1-u\le e^{-u}$ for $u\ge0$.

Let \begin{align}\label{defBcup}
    B:=\bigcup_{
\substack{ H\subseteq \G_N \\ 1\le|H|\le2K} } B_H.
\end{align}
For each \(1\le r\le2K\), the number of subsets \(H\subseteq \G_N\) with \(|H|=r\) is \(\binom{N}{r}\).
Since $\binom{N}{r}\le N^r\le N^{2K}$ for $1\le r\le2K$,
\begin{equation}\label{eq:number-families}
 \left| 
 \{H\subseteq \G_N:1\le|H|\le2K\} 
 \right| 
 =\sum_{r=1}^{2K}\binom{N}{r}
 \le2K N^{2K}.
\end{equation}
Then  
\begin{align}\label{eq:all-failure-bound}
     \mathbb P(B) & \nonumber \,\, \,\, \stackrel{\eqref{defBcup}}{\le}
     \sum_{\substack{ H\subseteq \G_N \\  1\le|H|\le2K} } \mathbb P(B_H)  \stackrel{\eqref{eq:fixed-failure-bound}}{\leq} 
\sum_{\substack{ H\subseteq \G_N \\ 1\le|H|\le2K} }  
\exp\left(-\frac{N}{4K^{2K+2}}\right) 
     \\ & \,\, \,\, \stackrel{\eqref{eq:number-families} }{\leq}
     2K N^{2K}\exp\left(-\frac{N}{4K^{2K+2}}\right)
    \overset{\scriptscriptstyle  N=K^{8K} }{<}  \frac14. 
     \end{align}

Hence, 
\begin{equation*}
 \mathbb P(B^c\cap C^c)
 =1-\mathbb P(B\cup C)
 \ge1-\mathbb P(B)-\mathbb P(C)
\stackrel{\eqref{eq:second-quarter}, \eqref{eq:all-failure-bound}}{>}
\frac12
>0.
\end{equation*} 
It follows that 
there exists a choice of \(A\) for which neither \(B\) nor \(C\) occurs.  By the definitions of $B$ and $C$, this set $A$ satisfies 
\eqref{eq:sparse}, i.e., 
$
|A|\stackrel{\eqref{defC}}{\le} 2 \mathbb E|A|  \stackrel{\eqref{eq:cardinality}}{=}
\frac{2N}{K}
$
and 
$
\bigcap_{g\in H}(A+g)\neq\varnothing
$
for every \(H\subseteq \G_N\) with \(1\le |H|\le 2K\) (see \eqref{defBcup}). 
The proof is complete. 
\end{proof}

Let $E,F$ be subsets of a set $\Omega$. Write
$E\triangle F:=(E\setminus F)\cup(F\setminus E)$
for their symmetric difference.
Note that 
\(
|\mathbf{1}_E(x)-\mathbf{1}_F(x)|=\mathbf{1}_{E \triangle F}(x). 
\)

Fix $K,N,A$ supplied by Lemma \ref{lem:sparse}. 
We next construct a translation-invariant submeasure
that assigns a large value to $A$ and uniformly small
values to the symmetric differences of its translates.

Set
\begin{align}\label{defDab}
D_{a,b}:=(A+a)\mathbin\triangle(A+b),
               \qquad a,b\in\G_N.
\end{align}
For $E\subseteq\G_N$, let 
\[
         \nu(E):=\inf
         \left\{       m\in\N\cup\{0\}:
         E\subseteq
                   \bigcup_{i=1}^mD_{a_i,b_i} \,\, 
\text{for some } a_i,b_i\in\G_N 
\right\}.
\] 
If no finite cover exists, set $\nu(E)=\infty$.  
Define
\begin{equation}\label{eq:phi}
                  \phi(E):=\min\left(1,\frac{\nu(E)}K\right).
\end{equation}
In particular, 
 $\phi(E)=1$ whenever $\nu(E)=\infty$.

\begin{lemma}\label{lem:submeasure}
The set function $\phi$ is a strictly positive, translation-invariant
submeasure on $\G_N$.  Moreover,
\begin{equation}\label{eq:phi-estimates}
                  \phi(A)=1,
          \qquad \phi(D_{a,b})\le K^{-1}\quad(a,b\in\G_N).
\end{equation}
\end{lemma}

\begin{proof}
Clearly $\phi(\varnothing)=\min\left(1,\frac{0}K\right)=0$. 
If $E\subseteq F\subseteq \G_N$, every cover of $F$ also covers $E$,
so $\nu(E)\le\nu(F)$ and hence $\phi(E)\le\phi(F)$. 
If $E\ne\varnothing$, then $\nu(E)\ge1$, with
$\nu(E)=\infty$ allowed. Then
\begin{equation}\label{phi(E)gefrac1K}
    \frac1K
    \overset{\scriptscriptstyle  \nu(E)\ge1, K \geq 2}{\le}
\phi(E)
    \stackrel{\eqref{eq:phi}}{\le}
    1
    \qquad(\varnothing\ne E\subseteq\G_N), 
\end{equation} 
which implies that  $\phi$ is strictly positive. 
Combining covers gives
$\nu(E\cup F)\le\nu(E)+\nu(F)$,
with the inequality automatic if either term on
the right is infinite.  
Since
\[
    \min\{1,p+q\}
    \le\min\{1,p\}+\min\{1,q\},
    \qquad p,q\in[0,\infty],
\]
definition~\eqref{eq:phi} yields
$\phi(E\cup F)\le\phi(E)+\phi(F)$.
Together with monotonicity and $\phi(\varnothing)=0$,
this shows that $\phi$ is a submeasure.

Observe that, $$ D_{a,b}+c=
\bigl((A+a)\triangle(A+b)\bigr)+c=
(A+a+c)\triangle(A+b+c)=
D_{a+c,b+c},$$ 
where $c\in \G_N$. 
If
$
E\subseteq  \bigcup_{i=1}^m D_{a_i,b_i},
$ 
then
$ 
E+c\subseteq  \bigcup_{i=1}^m D_{a_i+c,b_i+c}.
$ 
Thus,  
\(\nu(E+c)\le \nu(E)\). 
Applying the same argument with \(-c\) gives \(\nu(E)\le \nu(E+c)\), and therefore,
$
\nu(E+c)=\nu(E).
$
Consequently,
$ 
\phi(E+c)=\phi(E),
$
so \(\phi\) is translation invariant.

Since each $D_{a,b}$ covers itself, we have
$\nu(D_{a,b})\le1$, and so, 
$$\phi(D_{a,b})\stackrel{\eqref{eq:phi}}{\leq }
\min\{ 1,K^{-1}
\}\overset{\scriptscriptstyle  K \geq 2 }
{=} K^{-1}, \qquad a,b\in\G_N,$$ proving the second estimate
in \eqref{eq:phi-estimates}.  

Finally, we prove that $\phi(A)=1$. To this end, we show that any finite cover of $A$
by sets of the form $D_{a,b}$ must contain at least
$K$ sets.
Assume, towards a contradiction, that 
\begin{align}\label{0lemK}
A\subseteq  \bigcup_{i=1}^m D_{p_i,q_i}, \quad\text{for some }0\le m<K, \,\,  p_i,q_i\in\G_N.
\end{align}
The list of translation parameters
$0,p_1,q_1,\ldots,p_m,q_m$ has
$$ 1+2m\stackrel{\eqref{0lemK}}{\le} 1+2(K-1)= 2K-1< 2K
$$ entries.
By Lemma~\ref{lem:sparse}, there is a point $x\in A$
that belongs to both $A+p_i$ and $A+q_i$ for every $1\leq i \leq m$, 
and hence, 
$ 
x\notin D_{p_i,q_i}
\stackrel{\eqref{defDab}}{=}
(A+p_i)\mathbin\triangle(A+q_i)
$  for every $1\leq i \leq m$.  Therefore, 
\[
x\in A
\quad\text{but}\quad
x\notin \bigcup_{i=1}^m D_{p_i,q_i},
\]
which contradicts the covering assumption  \eqref{0lemK}. This implies that  $\nu(A)\ge K$, and so, $\phi(A)\stackrel{\eqref{eq:phi}}{=}1$.
\end{proof}


We now use the submeasure $\phi$ defined by  \eqref{eq:phi} to construct  a lattice quasi-norm
on a sequence space.

Let  
\begin{align*}
c_0(\G_N\times\Z)
 :=\left\{v=(v(r,\ell))_{(r,\ell)\in\G_N\times\mathbb Z}:
 \max_{r\in\G_N}|v(r,\ell)|\to 0
 \text{ as }|\ell|\to\infty\right\},
\end{align*}
equipped with the supremum norm 
\begin{align}\label{defvinftyoriginal}
\|v\|_\infty
:=\sup_{\substack{r\in\G_N\\ \ell\in\Z}}
|v(r,\ell)|, 
\end{align}
under which it is a Banach space.

For $v\in c_0(\G_N\times\Z)$, define
\begin{align} \label{defqv}       s_v(r):=\sup_{\ell\in\Z}|v(r,\ell)|,
 \qquad
             q(v):=\int_0^\infty
             \phi\bigl(\{r\in\G_N:        s_v(r)>t\}\bigr)\dd t.     
\end{align}
In particular, 
\begin{align*}
\|v\|_\infty
\stackrel{\eqref{defvinftyoriginal}, \eqref{defqv}}{=}
\mathop{\max}\limits_{r\in\G_N}s_v(r).  
\end{align*}
\begin{lemma}\label{lem:q}
The functional $q$ is a
lattice quasi-norm on $c_0(\G_N\times\Z)$
such that
\begin{equation}\label{eq:q-bounds}
 q(u+v)\le2\bigl(q(u)+q(v)\bigr),
 \qquad
 \frac1K\|v\|_\infty\le q(v)\le\|v\|_\infty.
\end{equation}
Consequently,  $(c_0(\G_N\times\Z),q)$ is complete and order continuous.
\end{lemma}

\begin{proof}
Since  
\begin{align}\label{setempty}
\{r\in\G_N:
 s_v(r)>t\}=\varnothing \quad \text{  for  every } t\ge \|v\|_\infty ,  
 \end{align}
 it follows that
\begin{align*}  
\frac1K\|v\|_\infty
& \nonumber \quad \,\,\,\,\,   = 
\int_0^{\|v\|_\infty}
\frac1K\,dt  
\stackrel{\eqref{phi(E)gefrac1K}}{\le}
   \int_0^{\|v\|_\infty}
     \phi\bigl(\{r:s_v(r)>t\}\bigr)\dd t 
\\ & \nonumber \overset{\phi(\varnothing)=0, \eqref{setempty}}{ =} 
\int_0^\infty
          \phi\bigl(\{r:s_v(r)>t\}\bigr)\dd t 
          \stackrel{\eqref{defqv}}{=}
          q(v)
        \\ & 
 \,\,\, \quad   \stackrel{\eqref{eq:phi}}{\le} \int_0^{\|v\|_\infty}
1\,dt = \|v\|_\infty.   
\end{align*} 
In particular, $q$ is finite and $ 
q(v)=0 $ iff $v=0. $
Observe that, 
for $u,v\in c_0(\G_N\times\Z)$, \[ \{r:s_{u+v}(r)>t\} \subseteq \{r:s_u(r)>t/2\}\cup\{r:s_v(r)>t/2\}. \] 
The monotonicity, subadditivity of $\phi$ (see Lemma \ref{lem:submeasure}) and \eqref{defqv} give  
$  q(u+v) 
\leq  2q(u)+2q(v). $
 Homogeneity and lattice monotonicity follow directly from the level-set
formula.
Thus,  $q$ is 
a lattice  quasi-norm on $c_0(\G_N\times\Z)$. 
 Since $q$ is equivalent to the supremum norm $\norm{\cdot}_{\infty}$  (see \eqref{eq:q-bounds}),   completeness and order continuity 
follow from the corresponding properties of $c_0(\G_N\times\Z)$. 
\end{proof}

Since  $q$ satisfies the quasi-triangle inequality
with constant $2$ (see Lemma ~\ref{lem:q}), we choose
$$p:=1/2,$$ so that
$
(2\cdot2)^p=2.
$
For $u, v\in c_0(\G_N\times\Z)$, define
\begin{equation}\label{eq:frink}
    \|v\|
    :=
    \left(
        \inf_{v=\sum_{j=1}^{J}v_j}
        \sum_{j=1}^{J}q(v_j)^p
    \right)^{1/p},
\end{equation}
where the infimum is taken over all finite
decompositions with integer $J\ge1$ and
$v_j\in c_0(\G_N\times\Z)$.  
The standard Frink--Aoki estimate \cite{Frink1937,Rolewicz1972}  yields
\begin{equation}\label{eq:b-estimate}
    \begin{aligned}
\|u+v\|^p\le\|u\|^p+\|v\|^p,\qquad \frac1{16}q(v)\le\|v\|\le q(v).
    \end{aligned}
\end{equation}
The proof of \cite[Proposition~1.3]{Brudnyi2015}
shows that $\norm{\cdot}$ is lattice monotone.
Its equivalence to $q$ then yields completeness
and order continuity.

\begin{lemma}\label{lem:renorming} 
Fix $j\in\G_N$ and choose an integer $c_r$ for each
$r\in\G_N$.  
Define  a linear bijection  
$T: c_0(\G_N\times\Z) \to c_0(\G_N\times\Z)$ by 
\begin{equation}\label{eq:fiber-permutation}
    (Tv)(r,\ell)
    :=v(r+j\pmod N,\ell+c_r),
    \qquad (r,\ell)\in\G_N\times\Z.
\end{equation}
Then $T$ satisfies
\begin{equation}\label{eq:fiber-invariance}
    q(Tv)=q(v),
    \qquad
    \norm{Tv}=\norm{v},
    \qquad v\in c_0(\G_N\times\Z).
\end{equation}
\end{lemma}

\begin{proof}

For every $\varepsilon>0$, a sequence in
$c_0(\G_N\times\Z)$ has only finitely many coordinates
whose absolute values are at least $\varepsilon$. 
For every $\varepsilon>0$,
\[
    | \{(r,\ell):|(Tv)(r,\ell)|\ge\varepsilon\}
    |
    =
    | \{(r,\ell):|v(r,\ell)|\ge\varepsilon\} 
    |
    <\infty,
\]
and so, $Tv\in c_0(\G_N\times\Z)$.
The inverse of $T$ is given by
\[
    (T^{-1}v)(r,\ell)
    =v(r-j,\ell-c_{r-j}),
    \qquad (r,\ell)\in\G_N\times\Z,
\]
where the first coordinate is taken modulo $N$. 
Thus, $T$ is a linear bijection of
$c_0(\G_N\times\Z)$ onto itself.

For every $r\in\G_N$,
\begin{align}\label{sTvsv}
    s_{Tv}(r)  \stackrel{\eqref{defqv}, \eqref{eq:fiber-permutation}}{=}   \sup_{\ell\in\Z}|v(r+j,\ell+c_r)|
    =s_v(r+j),
\end{align}
since $\ell\mapsto\ell+c_r$ is a bijection of $\Z$. 
Since 
\[
    \{r\in\G_N:s_{Tv}(r)>t\}
   \stackrel{\eqref{sTvsv}}{ =}
    \{r\in\G_N:s_v(r)>t\}-j, 
\]
it follows from 
 the translation invariance of $\phi$  (see Lemma \ref{lem:submeasure}) that 
\begin{align}\label{qTv=qv}
    q(Tv)   & \nonumber \stackrel{\eqref{defqv}}{=}
    \int_0^\infty
      \phi\bigl(\{r\in\G_N:s_v(r)>t\}-j\bigr)\,dt
     \\ &  \,\, = \int_0^\infty
      \phi\bigl(\{r\in\G_N:s_v(r)>t\}\bigr)\,dt
   \stackrel{\eqref{defqv}}{=}q(v).
\end{align}

For any finite decomposition $v=\sum_{i=1}^{J}v_i$,
linearity gives $Tv=\sum_{i=1}^{J}Tv_i$.
Hence, 
\[
    \norm{Tv}^{p}
\stackrel{\eqref{eq:frink}}{\le}
\sum_{i=1}^{J}q(Tv_i)^p
    \stackrel{\eqref{qTv=qv}}{=}
    \sum_{i=1}^{J}q(v_i)^p.
\]
Taking the infimum over all finite decompositions
of $v$ yields
$
    \norm{Tv}^{p}\le\norm{v}^{p}.
$
Applying the same argument to $T^{-1}$ gives
$ 
    \norm{v}^{p}
    =\norm{T^{-1}(Tv)}^{p}
    \le\norm{Tv}^{p}.
$ 
Consequently, $\norm{Tv}=\norm{v}$, proving
\eqref{eq:fiber-invariance}.

\end{proof}

Fix $\delta>0$ and  identify the elements of $\G_N$ with their
representatives in $\{0,\ldots,N-1\}$. Every $x\in\R$ has a unique representation
\[
            x=t+\delta(r+N\ell),
 \qquad 0\le t<\delta,\quad r\in\G_N,\quad \ell\in\Z.
\]
For a measurable function $f: \R\to\R$,  and each $t\in[0,\delta)$,
define the sequence
$V_f(t)=(V_f(t)(r,\ell))_{(r,\ell)\in\G_N\times\Z}$
by
\begin{equation}\label{eq:fiberization}
    V_f(t)(r,\ell)
    :=f\bigl(t+\delta(r+N\ell)\bigr).
\end{equation}

The preceding construction defines a function lattice
$(Y,\norm{\cdot}_Y)$ depending on $K$, $N$, $A$ supplied by Lemma \ref{lem:sparse} and $\delta>0$.

\begin{proposition}\label{prop:Y} 

Let
\[
Y:=
\left\{
f\in L^0(\R):
V_f(t)\in c_0(\G_N\times\Z)
\text{ for a.e. }t\in[0,\delta),
\quad \|f\|_Y<\infty
\right\},
\]
where
\begin{equation}\label{eq:Y-norm}
    \|f\|_Y
    :=
    \left(
       \frac1\delta\int_0^\delta
        \|V_f(t)\|^p\,dt
    \right)^{1/p}.
\end{equation}
The space $(Y, \norm{\cdot}_Y)$ is an order-continuous $p$-Banach function lattice, where $p=1/2$.   Every bounded, compactly supported 
function $f$ belongs to $Y$ and satisfies \begin{align}\label{eq:Y-sup-bound}
\norm{f}_{Y}\le\norm{f}_\infty.
\end{align}
Moreover, every real
translation is a surjective isometry of $Y$. In particular, 
$ (Y,\norm{\cdot}_Y^p)
$ belongs to
Brudnyi's class $\cL$. 
\end{proposition}

\begin{proof}
Under the natural identification $f\mapsto V_f$,
$Y$ is isometrically isomorphic to 
\[
                 L^p\left([0,\delta), \big( 
       c_0(\G_N\times\Z), \norm{\cdot}            \big);\frac{\dd  t}{\delta}
                 \right).
\]
This space is complete because
$(c_0(\G_N\times\Z),\norm{\cdot})$ is a
$p$-Banach space; see \cite[page 6]{Pisier2016}.
Hence $Y$ is complete. 
The $p$-triangle inequality for $\norm{\cdot}_Y$
follows directly from ~\eqref{eq:b-estimate}.

Let $g\in Y$ and let $f$ be measurable with
$|f|\le|g|$ almost everywhere.
Since $\G_N\times\Z$ is countable and for each fixed $(r,\ell)\in\G_N\times\Z$ and  almost every $t$, 
\[
 |V_f(t)(r,\ell)|
    \stackrel{\eqref{eq:fiberization} }{=}
    |
    f\bigl(t+\delta(r+N\ell)\bigr)
    |
    \leq 
    |
    g\bigl(t+\delta(r+N\ell)\bigr)
    |
      \stackrel{\eqref{eq:fiberization} }{=}
      |V_g(t)(r,\ell)|,  
\] 
it follows that 
$ 
      |V_f(t)|\le|V_g(t)|
  $   for almost every $t.$ 
As $V_g(t)\in c_0(\G_N\times\Z)$ almost everywhere,
this domination implies that
$V_f(t)\in c_0(\G_N\times\Z)$ almost everywhere.
By lattice monotonicity,
$\|V_f(t)\|\le\|V_g(t)\|$.
Therefore, \eqref{eq:Y-norm} gives $f\in Y$ and
$\|f\|_Y\le\|g\|_Y$.

Now let $f$ be bounded by $\norm{f}_{\infty}$ and compactly supported.
For almost every $t\in[0,\delta)$, $V_f(t)$ is finitely supported, and 
$\|V_f(t)\|
\stackrel{\eqref{eq:q-bounds}, \eqref{eq:b-estimate}}{\le}
\|V_f(t)\|_\infty
\stackrel{\eqref{eq:fiberization} }{\le}
\|f\|_\infty$.
Hence $f\in Y$ and $\norm{f}_Y\le\norm{f}_\infty$
by \eqref{eq:Y-norm}.

We next prove order continuity.
Let $(f_j)\subseteq Y$ satisfy
$0\le f_j\downarrow0$ almost everywhere.
For each fixed $(r,\ell)\in\G_N\times\Z$, we have
\[
0\le  V_{f_j}(t)(r,\ell)
\stackrel{\eqref{eq:fiberization} }{=} f_j\bigl(t+\delta(r+N\ell)\bigr)\downarrow0
\]
for almost every $t\in[0,\delta)$.
Since the coordinate set $\G_N\times\Z$ is countable, $V_{f_j}(t)\downarrow0$ coordinatewise for
almost every $t$.
Order continuity of $\norm{\cdot}$ therefore yields
$\|V_{f_j}(t)\|\to0$ for almost every $t$. 
Since 
$ 
    \|V_{f_j}(t)\|^p
    \le\|V_{f_1}(t)\|^p  
$
and 
$f_1\in Y$, 
applying the dominated convergence theorem
in \eqref{eq:Y-norm} gives  $\|f_j\|_Y\to0$.

It remains to prove translation invariance.
Fix $h\in\R$ and write
\begin{align}\label{defhdecom}
    h=a\delta+s,
    \qquad a\in\Z,\quad 0\le s<\delta.
\end{align}
Fix $t\in[0,\delta)$. Then,  
\begin{align*}
 V_{\tau_hf}(t)(r,\ell)
 &\stackrel{\eqref{eq:fiberization} }{=}
 \tau_hf(t+\delta(r+N\ell))
 =
 f\bigl(t+\delta(r+N\ell)+h\bigr)\\
 &
 \stackrel{\eqref{defhdecom}}{=}
  f\big(
 t+s-\one_{[\delta-s,\delta)}(t)\delta
 +\delta(r+N\ell+a+\one_{[\delta-s,\delta)}(t)
 )
 \big).
\end{align*}
 For each $r\in\{0,\ldots, N-1\}$,
write  the following unique decomposition 
\[
 r+a+\one_{[\delta-s,\delta)}(t)=r'+Nd_r,
 \]
where
$r'=r+a+\one_{[\delta-s,\delta)}(t) \pmod N 
\in\{0,\ldots,N-1\},$
$d_r\in\mathbb Z.$ 
We obtain
\begin{align*}
 V_{\tau_hf}(t)(r,\ell)
& \, \, =f\bigl(
t+s-\one_{[\delta-s,\delta)}(t)\delta
+\delta(r'+N(\ell+d_r))\bigr)
\\ & \stackrel{\eqref{eq:fiberization} }{=}
V_f(t+s-\one_{[\delta-s,\delta)}(t)\delta)
(r',\ell+d_r)
.
\end{align*}
This is precisely a coordinate permutation of the
form \eqref{eq:fiber-permutation}.
Consequently, for each fixed $t$, there is a
corresponding operator $T$ such that $  V_{\tau_hf}(t)(r,\ell) \stackrel{ \eqref{eq:fiber-permutation}}{=} 
[T V_f(t+s-\one_{[\delta-s,\delta)}(t)\delta) ] (r,\ell)$, and so, 
\[
V_{\tau_h f}(t)=T V_f(t+s-\one_{[\delta-s,\delta)}(t)\delta).
\]
Since $f\in Y$ and the map $t\mapsto t+s-\one_{[\delta-s,\delta)}(t)\delta$ preserves Lebesgue measure,
$V_f(t+s-\one_{[\delta-s,\delta)}(t)\delta)$ belongs to $c_0(\G_N\times\Z)$ for almost
every $t$.
The same holds for $V_{\tau_h f}(t)$ by Lemma ~\ref{lem:renorming}. 
Hence,
\[
\begin{aligned}
\frac1\delta\int_0^\delta
    \|V_{\tau_h f}(t)\|^p\,\dd t
&\stackrel{\eqref{eq:fiber-invariance}}{=}\frac1\delta\int_0^\delta
    \|V_f(
    t+s-\one_{[\delta-s,\delta)}(t)\delta
    )\|^p\,\dd t\\ 
&\,\,=\frac1\delta 
\left( 
\int_0^{\delta-s}\|V_f(t+s)\|^p\,\dd t
  +\int_{\delta-s}^{\delta}\|V_f(t+s-\delta)\|^p\,\dd t
  \right)\\
&\,\,=\frac1\delta \left( 
\int_s^\delta\|V_f(u)\|^p\,\dd u
  +\int_0^s\|V_f(u)\|^p\,\dd u
  \right)\\
&\,\,=\frac1\delta \int_0^\delta\|V_f(u)\|^p\,\dd u
\stackrel{\eqref{eq:Y-norm}}{=}
\|f\|_Y^p.
\end{aligned}
\]
This implies that  $\tau_h f\in Y$ and
$\|\tau_h f\|_Y=\|f\|_Y$.
Since $\tau_{-h}$ is the inverse of $\tau_h$,
every translation is a surjective isometry of $Y$.
\end{proof}

\begin{lemma}\cite[Proposition~1.2(a)]{Brudnyi2015}\label{lem:Y-local-measure}
Convergence in $Y$ implies convergence
in measure on every bounded interval. 
\end{lemma}

\begin{lemma}\cite[Proposition~1.10]{Brudnyi2015}\label{lem:strong-cont-Y}
Translations are strongly continuous in $Y$:
$\|\tau_hf-f\|_Y\to0$ as $h\to0$ for every $f\in Y$.
\end{lemma}

\subsection{Construction of a periodic function and  a test function} 

Retain the integer $K\ge2$ and the corresponding
$N$ and $A\subseteq\G_N$ from
Lemma~\ref{lem:sparse}. 
Fix a Lipschitz function $\psi\colon\R\to[0,1]$
with Lipschitz constant $L$ such that
$\operatorname{supp}\psi\subseteq Q$ for a compact
interval $Q$ and $\psi=1$ on $[0,1]$.
Choose \begin{align}\label{rangtedelta}
    0<\delta\le 1/N,
\end{align}
and let $Y$ be the corresponding space from
Proposition~\ref{prop:Y}.

For each $x\in\R$, let $k\in\Z$ be the unique integer
such that $k\delta\le x<(k+1)\delta$. 
Define the
 function $e_A$ by 
\begin{equation}\label{eq:pattern}
 e_A(x)=\one_A \, (k\bmod N)
=
 \begin{cases}
 1,& k\bmod N\in A,\\
 0,& k\bmod N\notin A.
 \end{cases}
\end{equation}
Note that $e_A$ is 
$N\delta$-periodic. 
Indeed, \[
\begin{aligned}
e_A(x+N\delta)
=\one_A((k+N)\bmod N)=\one_A(k\bmod N)=e_A(x).
\end{aligned}
\]
For a weight $w>0$, put
\begin{align}\label{defux}
     u(x):=(w^{-1}\psi e_A)
     (x).
\end{align}
Since $|u|\le w^{-1}$ and $\operatorname{supp}u\subseteq Q$,
the function $u$ is bounded and compactly supported.
Hence $u\in Y$ by Proposition~\ref{prop:Y}.

\begin{proposition}\label{prop:critical}
The function $u$ satisfies
\begin{equation}\label{eq:critical-size}
                 \frac1{16w}\le\|u\|_Y\le\frac1w.
\end{equation}
Moreover, for every $h\in\R$, 
\begin{equation}\label{eq:critical-modulus}
 w\|\tau_hu-u\|_Y
 \le 2^{1/p-1} \bigl(L|h|+K^{-1}\bigr),
\end{equation} 
where $L$ is the Lipschitz constant of $\psi$.
\end{proposition}

\begin{proof}
For every $t\in[0,\delta)$, $r=0,\ldots,N-1$, and
$\ell\in\Z$,   we have 
\begin{align*}
 V_u(t)(r,\ell)
& \stackrel{\eqref{eq:fiberization}}{=} u\bigl(t+\delta(r+N\ell)\bigr)
\stackrel{\eqref{defux}}{=}
\frac1w\psi\bigl(t+\delta(r+N\ell)\bigr)
e_A\bigl(t+\delta(r+N\ell)\bigr)\\ & 
\stackrel{\eqref{eq:pattern}}{=}
\frac1w\psi\bigl(t+\delta(r+N\ell)\bigr)
\one_A \, (r),
\end{align*} 
where we used the fact that 
$t+\delta(r+N\ell) \in [\delta(r+N\ell), \delta(r+N\ell+1) )$ and  $(r+N\ell) \bmod N=r$. 
Since $0\le\psi\le1$,  
$
\sup_{\ell\in\Z}|V_u(t)(r,\ell)|
    \le w^{-1}\one_A(r).
$ 
On the other hand, $0\le t+\delta r<N\delta
\stackrel{\eqref{rangtedelta}}{\le}
1$,
so $\psi(t+\delta r)=1$. Thus, \(|V_u(t)(r,0)|=\frac1w\one_A(r).\)
This  shows that
\begin{equation}\label{supVut}
    \sup_{\ell\in\Z}|V_u(t)(r,\ell)|
    =w^{-1}\one_A(r).
\end{equation}
 Hence,  \[
\begin{aligned}
 q(V_u(t))
 & \quad  \stackrel{\eqref{defqv}}{=}
 \int_0^\infty
 \phi\left(\left\{r\in\G_N:
 \sup_{\ell\in\Z}|V_u(t)(r,\ell)|>\lambda\right\}\right)
 \,\dd\lambda
 \\
 & \quad 
 \stackrel{\eqref{supVut}}{=}
\int_0^\infty
\phi\bigl(\{r\in\mathbb G_N:w^{-1}\one_A(r)>\lambda\}\bigr)
\,d\lambda
\\
 & \overset{0<w\lambda<1}{=}
 \int_0^{w^{-1}}\phi(A)\,\dd\lambda
=w^{-1}\phi(A)
\stackrel{\eqref{eq:phi-estimates}}{=}w^{-1} . 
\end{aligned}
\] 
Estimate \eqref{eq:b-estimate}
gives
\[
 \frac1{16w} =\frac1{16} q(V_u(t))
 \le\norm{V_u(t)}\le q(V_u(t)) =
 \frac1w,
 \qquad t\in[0,\delta). 
\]
Taking $p$th powers, averaging over $[0,\delta)$,
and then taking $p$th roots gives 
\eqref{eq:critical-size} by ~\eqref{eq:Y-norm}.

Fix $h\in\R$.
We next decompose $\tau_hu-u$ into two terms
corresponding to the changes in $\psi$ and $e_A$.
For $x\in\R$, define
\[
 F_h(x):= \big( 
 \psi(x+h)-\psi(x)
  \big)
  e_A(x+h),
 \quad
 G_h(x):=\psi(x) \big(  e_A(x+h)-e_A(x)
  \big).
\]
Then, 
\begin{align}\label{w(tau_hu-u)(x)decom}
 w(\tau_hu-u)(x)
 \stackrel{\eqref{defux}}{=}
 \psi(x+h)e_A(x+h)-\psi(x)e_A(x)=F_h(x)
   +G_h(x).
\end{align}
Since \(|F_h(x)|\le1\),\(|G_h(x)|\le1\), \(\operatorname{supp}F_h\subseteq Q\cup(Q-h)\) and \(\operatorname{supp}G_h\subseteq Q\), it follows from Proposition ~\ref{prop:Y}  that 
 $F_h, G_h\in Y$. 
The Lipschitz
condition on $\psi$   yields 
\begin{align}\label{normFhYp}
 \|F_h\|_Y\stackrel{\eqref{eq:Y-sup-bound}}{\le } \|F_h\|_\infty\le  L|h|.
\end{align}

To estimate $G_h=\psi(x) \big(  e_A(x+h)-e_A(x)
  \big)$, write $h\stackrel{\eqref{defhdecom}}{=} a\delta+s$, where
$a\in\Z$ and $0\le s<\delta$. 
For every $t\in[0,\delta)$, 
we have 
\begin{align}\label{t+s-c(t)delta}
0\le  t+s-\one_{[\delta-s,\delta)}(t)\delta<\delta, \quad  
 \delta(r+N\ell)
 \le t+\delta(r+N\ell)
<\delta(r+N\ell+1).
\end{align}
In particular, 
\begin{align}\label{1A}
e_A\bigl(t+\delta(r+N\ell)\bigr)
 \stackrel{\eqref{eq:pattern}}{=}
 \one_A\big(
 (r+N\ell)\bmod N
 \big)
 =\one_A(r).
\end{align}
Since 
\begin{align*}
\delta\big(
r+N\ell+a+\one_{[\delta-s,\delta)}(t) 
\big) &\stackrel{\eqref{t+s-c(t)delta}}{\leq } 
t+s-\one_{[\delta-s,\delta)}(t) \delta
+\delta\big(
r+N\ell+a+ \one_{[\delta-s,\delta)}(t) 
\big)
 \\ &\,\, =t+\delta(r+N\ell)+a\delta+s
=
t+\delta(r+N\ell)+h
\\ &  
\stackrel{\eqref{t+s-c(t)delta}}{<}\delta\big(
r+N\ell+a+\one_{[\delta-s,\delta)}(t) +1
\big), 
\end{align*}
the definition 
of $e_A$ 
gives
\begin{align}\label{1A-a-c(t)}
e_A\bigl(t+\delta(r+N\ell)+h\bigr)
&\nonumber \stackrel{\eqref{eq:pattern}}{=}
\one_A\big( 
(
r+N\ell+a+\one_{[\delta-s,\delta)}(t) 
) \bmod N
\big)
\\ & \,\, = \one_A(r+a+\one_{[\delta-s,\delta)}(t))
 =\one_{A-a-\one_{[\delta-s,\delta)}(t)}(r).
\end{align}
Using $0\le\psi\le1$ and the identity
$|\one_E-\one_F|=\one_{E\triangle F}$, we obtain
\[
\begin{aligned}
 |V_{G_h}(t)(r,\ell)|
&~\,\,\,
\stackrel{\eqref{eq:fiberization}}{=} \left| G_h\bigl(t+\delta(r+N\ell)\bigr)
\right|\\ &
~~\,\, \,\, \,
=\left| \psi\bigl(t+\delta(r+N\ell)\bigr) \big(  e_A\bigl(t+\delta(r+N\ell)+h\bigr)-e_A\bigl(t+\delta(r+N\ell)\bigr)
  \big)
  \right|
\\&\stackrel{\eqref{1A}, 
\eqref{1A-a-c(t)}}{\leq}
\left|\psi\bigl(t+\delta(r+N\ell)\bigr)\right|
 \left|\one_{A-a-\one_{[\delta-s,\delta)}(t) }(r)-\one_A(r)\right|\\
 &~~\,\, \,\, \,  \le\one_{(A-a-\one_{[\delta-s,\delta)}(t) )\triangle A}(r). 
\end{aligned}
\] 
For fixed $t\in[0,\delta)$, the preceding bound gives
\[ \sup_{\ell\in\Z}|V_{G_h}(t)(r,\ell)|
 \le\one_{(A-a-\one_{[\delta-s,\delta)}(t))\triangle A}(r).
\]
Consequently, for $0<\lambda<1$,
\begin{align}\label{setsubseteqrelation}
 \left\{r\in\G_N:
 \sup_{\ell\in\Z}|V_{G_h}(t)(r,\ell)|>\lambda\right\}
 \subseteq (A-a-\one_{[\delta-s,\delta)}(t))\triangle A.
\end{align}
For $\lambda\ge1$, 
the  set $\left\{r\in\G_N:
 \sup_{\ell\in\Z}|V_{G_h}(t)(r,\ell)|>\lambda\right\}$ is empty. 
By  the monotonicity of $\phi$ (see Lemma \ref{lem:submeasure}),
\[
\begin{aligned}
 q(V_{G_h}(t))
 &\stackrel{\eqref{defqv}}{=}
 \int_0^\infty
 \phi\left(\left\{r\in\G_N:
 \sup_{\ell\in\Z}|V_{G_h}(t)(r,\ell)|>\lambda\right\}\right)
 \,\dd\lambda\\
 &
\stackrel{\eqref{setsubseteqrelation}}{\le}
 \int_0^1
 \phi\bigl((A-a-\one_{[\delta-s,\delta)}(t) )\triangle A\bigr)\,\dd\lambda\\
 &\,\, =\phi\bigl((A-a-\one_{[\delta-s,\delta)}(t))\triangle A\bigr)\stackrel{\eqref{eq:phi-estimates}}{\le }
 K^{-1}.
\end{aligned}
\]
Since $\norm{V_{G_h}(t)}
\stackrel{\eqref{eq:b-estimate}}
 {
 \le}
 q(V_{G_h}(t))$, it follows from  \eqref{eq:Y-norm} that
\begin{align}\label{normGhYp}
 \norm{G_h}_Y^p \leq K^{-p}.
\end{align}

The $p$-triangle inequality now yields
\begin{align*}
\Big(  w\norm{\tau_hu-u}_Y
\Big)^p \stackrel{\eqref{w(tau_hu-u)(x)decom}}{=}
 \norm{F_h+G_h}_Y^p
 \stackrel{\text{Prop.\ref{prop:Y}}}{\le}
 \norm{F_h}_Y^p+\norm{G_h}_Y^p
\stackrel{\eqref{normFhYp}, 
 \eqref{normGhYp} }
 {\le}
 (L|h|)^p+K^{-p}.
\end{align*}
Using
$(a^p+b^p)^{1/p}\le2^{1/p-1}(a+b)$ for $a,b\ge0$,
we conclude that
\[
 w\norm{\tau_hu-u}_Y
 \le\bigl((L|h|)^p+K^{-p}\bigr)^{1/p}
 \le 2^{1/p-1} \bigl(L|h|+K^{-1}\bigr).
\]
This proves \eqref{eq:critical-modulus}.
\end{proof}

We next establish a uniform estimate for the
$Y$-norms of indicators of sets of small measure.

\begin{lemma}\label{lem:small-set}
Let
$Q\subseteq\R$ be a compact interval.
Then, for every measurable set $E\subseteq Q$,
\[
\|\one_E\|_Y^p\le\frac{|E|}{\delta}.
\]
\end{lemma}

\begin{proof}
Let $E\subseteq Q$ be measurable.
Since $\one_E$ is bounded and compactly supported,
Proposition~\ref{prop:Y} gives $\one_E\in Y$.
Since 
\[
\|V_{\one_E}(t)\|^p
\stackrel{\eqref{eq:q-bounds}, \eqref{eq:b-estimate} }{\leq}
\|V_{\one_E}(t)\|_\infty^p
=\sup_{k\in\Z}\one_E(t+\delta k)
\stackrel{\eqref{defvinftyoriginal}}{=} 
\|V_{\one_E}(t)\|_\infty, 
\]
it follows that 
\begin{align*}
\|\one_E\|_Y^p
&\stackrel{\eqref{eq:Y-norm}}{=}
\frac1\delta\int_0^\delta
  \|V_{\one_E}(t)\|^p\,\dd t
  \le
  \frac1\delta\int_0^\delta
  \sup_{k\in\Z}\one_E(t+\delta k)\,\dd t\\
&\, \, \le\frac1\delta\int_0^\delta
  \sum_{k\in\Z}\one_E(t+\delta k)\,\dd t
  =\frac1\delta\sum_{k\in\Z}
  \int_0^\delta\one_E(t+\delta k)\,\dd t\\
&\,\, =\frac1\delta\sum_{k\in\Z}
  \int_{\delta k}^{\delta(k+1)}\one_E(x)\,\dd x
=\frac1 \delta
\int_{\R}\one_E(x)\,\dd x 
=\frac{|E|}{\delta}.
\end{align*}
\end{proof}

\subsection{Construction of the function lattice $X$}

Recall that we have fixed that $\psi$  is  a Lipschitz function on $\R$, with the  Lipschitz constant $L$, such that $0\le\psi\le1$,
$\operatorname{supp}\psi\subseteq  Q$ for a compact interval $Q$,
and $\psi=1$ on $[0,1]$.

We now construct an order-continuous,
translation-invariant $p$-Banach function lattice
$X$ on $\R$.
We first choose spaces
$Y_n=Y(K_n,N_n,A_n,\delta_n)$ and positive weights
$w_n$ recursively, and then define $X$ by imposing
a weighted $c_0$ condition.

We choose the parameters recursively. 
Start with $w_1=1$. For $n\ge2$, the spaces $Y_1,\ldots,Y_{n-1}$ and the
positive weights $w_1,\ldots,w_n$ have already been fixed. 
For $n=1$, choose any $0<d_1\le 2^{-1} $.
Choose
\begin{align}\label{defdnrange}
    0<d_n\le
    \min\left\{
        2^{-n},
        \min_{1\le m<n}
        \delta_m
        \left(2^{-n}\frac{w_n}{w_m}\right)^p
    \right\}, \qquad  n\geq 2.
\end{align}
By Lemma~\ref{lem:small-set}, every measurable
set $E\subseteq Q$ with $|E|<d_n$ satisfies
\begin{equation}\label{eq:past-small}
    \norm{\one_E}_{Y_m}
    \le
    \left(\frac{|E|}{\delta_m}\right)^{1/p}
    \le
    2^{-n}\frac{w_n}{w_m},
    \qquad \forall m<n.
\end{equation}
Choose an integer $K_n$ such that
\begin{equation}\label{eq:K-choice}
    K_n>\max\left\{2^n,\frac{4|Q|}{d_n}\right\}.
\end{equation}
Apply Lemma \ref{lem:sparse} to obtain an integer
$N_n:=K_n^{8 K_n}$ and a set $A_n \subseteq\mathbb G_{N_n}$ such that 
\begin{align}\label{fracAnNnlefrac2Kn}
|A_n|\le\frac{2N_n}{K_n}<N_n.
\end{align}
Choose
\begin{align}\label{choosePn}
    0< \delta_n <\frac{d_n}{4 N_n} \,\,\,   \left(\stackrel{\eqref{defdnrange}}{\le} \frac{ 2^{-n}}{4 N_n}< 
    \frac{ 1}{ N_n}
    \right).
\end{align}
With these parameters fixed, let
\[
    Y_n:=Y(K_n,N_n,A_n,\delta_n)
\]
be the $p$-Banach function lattice constructed in
Proposition~\ref{prop:Y}, where $p=1/2$.  
To complete the $n$th stage, choose $w_{n+1}$ such that
\begin{equation}\label{eq:weight-choice}
                       0<w_{n+1}\le2^{-n} w_n.
\end{equation}
In particular, $w_n\downarrow0$ and $w_n\le1$ (since $w_1=1$).

We now combine the spaces $Y_m$ by imposing
a weighted $c_0$ condition:
\begin{equation}\label{eq:X}
 X:=\left\{f\in\bigcap_{m\ge1}Y_m:
      w_m\|f\|_{Y_m}\to 0 \,\, \text{as}\,\, m\to 
     \infty
     \right\},  \quad  \|f\|_X:=\sup_{m\ge1} w_m\|f\|_{Y_m}.
\end{equation}

\begin{proposition}
\label{prop:X}
The space $X$ is an  order-continuous $p$-Banach function lattice.
Every real translation is a surjective isometry of $X$,
and the translation group is strongly continuous on $X$.  Moreover, every bounded compactly supported \(g\) belongs to \(X\) and satisfies \(\|g\|_X\le\|g\|_\infty\).
\end{proposition}

\begin{proof}
Homogeneity, the ideal property, and the $p$-triangle inequality follow
coordinatewise from the  properties of the $Y_m$. Definiteness follows from $w_1>0$ and the
definiteness of $\norm{\cdot}_{Y_1}$.
Thus $X$ is a $p$-normed function lattice.

We first prove completeness. 
Let $(f_j)$ be a Cauchy sequence in $X$.
For each fixed $m\ge1$, the definition of $\norm{\cdot}_X$
gives
\[
 \norm{f_j-f_k}_{Y_m}
 \le w_m^{-1}\norm{f_j-f_k}_X
 \to 0
 \qquad\text{as }j,k\to\infty.
\]
Since $w_m>0$ is fixed, $(f_j)$ is Cauchy in $Y_m$.
By the completeness of $Y_m$, there exists $g_m\in Y_m$
such that
\begin{align*}
 \norm{f_j-g_m}_{Y_m}\to 0
 \qquad\text{as }j\to\infty.
\end{align*}
By Lemma \ref{lem:Y-local-measure},
$f_j\to g_m$ locally in measure for each $m\ge1$. 
Uniqueness of the limit in measure gives
$g_m=g_1$ almost everywhere on $\R$. 
Denote the common equivalence class of the functions $g_m$
by $f$. 
Then,  $ 
    f\in\mathop{\bigcap}\limits_{m\ge1}Y_m
$ 
    and
$\norm{f_j-f}_{Y_m}\to0$ for every $m\ge1$.

Let $\varepsilon>0$. Choose  $j_0$ such that
$\norm{f_j-f_k}_X<\varepsilon$ for all $j,k\ge j_0$. 
Consequently,
\[
 w_m\norm{f_j-f_k}_{Y_m}\stackrel{\eqref{eq:X}}{\le} \norm{f_j-f_k}_X < \varepsilon
 \qquad \forall m\ge1,\ j,k\ge j_0.
\]
Fixing $j\ge j_0$ and letting $k\to\infty$, we obtain
$ 
 w_m\norm{f_j-f}_{Y_m}\le\varepsilon
$ for all $ m\ge1$.
Hence, 
\begin{align}\label{supmge1wmnormfjfYmlevarepsilon}
 \sup_{m\ge1}w_m\norm{f_j-f}_{Y_m}\le\varepsilon, \qquad j\ge j_0.
\end{align}
For any such \(j\), the \(p\)-triangle inequality and \(f_j\in X\) imply
\begin{align*}
\mathop{\limsup}\limits_{m\to\infty}(w_m\norm{f}_{Y_m})^p
 \stackrel{ \eqref{supmge1wmnormfjfYmlevarepsilon} 
 }{\le}
 \varepsilon^p.
\end{align*}
Since  $\varepsilon>0$ is arbitrary, 
it follows that
\begin{align*}
    w_m\norm{f}_{Y_m}\to0  \qquad\text{as }m\to\infty.
    \end{align*}
Therefore, $f\in X$  and 
\[
 \norm{f_j-f}_X
   \stackrel{\eqref{eq:X}
 }{=}\sup_{m\ge1}w_m\norm{f_j-f}_{Y_m}
   \stackrel{\eqref{supmge1wmnormfjfYmlevarepsilon}
 }{\le}
 \varepsilon,
 \qquad j\ge j_0.
\]
Hence,  $f_j\to f$ in $X$, proving completeness.

Combining the $c_0$ condition in \eqref{eq:X}
with lattice monotonicity and the order continuity
of each $(Y_m, \norm{\cdot}_{Y_m})$, we conclude that $X$ is order continuous.
Let $f\in X$ and $h\in\R$.
Since translations are isometries on each $Y_m$ (see Proposition \ref{prop:Y}),
 it follows that $\tau_hf\in X$ and
\[
 \norm{\tau_hf}_X
 \stackrel{\eqref{eq:X}
 }{=}\sup_{m\ge1}w_m\norm{\tau_hf}_{Y_m}
 =\sup_{m\ge1}w_m\norm{f}_{Y_m}
 \stackrel{\eqref{eq:X}
 }{=} 
 \norm{f}_X.
\]
Surjectivity of the map $\tau_h:X\to X$  is immediate. 

To prove strong continuity of translations on $X$, fix $f\in X$ and
$\varepsilon>0$. 
Since $w_m\norm{f}_{Y_m}\stackrel{\eqref{eq:X}}{\to} 0$ 
as $m\to\infty$,
choose an integer $M\ge2$ such that
\begin{align}\label{supwmnormfYmvare}
 \sup_{m\ge M}w_m\norm{f}_{Y_m}
 <2^{-1/p}\varepsilon.
\end{align}
For every $h\in\R$, the $p$-triangle inequality and
the isometry of translations on $Y_m$ (see Proposition \ref{prop:Y}) give
\[
\begin{aligned}
\big(  w_m \norm{\tau_hf-f}_{Y_m}
\big)^p
 \le w_m^p
\bigl(\norm{\tau_hf}_{Y_m}^p+\norm{f}_{Y_m}^p\bigr)=2w_m^p\norm{f}_{Y_m}^p.
\end{aligned}
\]
Consequently,
\begin{align}\label{1.1}
 \sup_{m\ge M}w_m\norm{\tau_hf-f}_{Y_m}
 \le2^{1/p}\sup_{m\ge M}w_m\norm{f}_{Y_m}
\stackrel{\eqref{supwmnormfYmvare}}{<}
\varepsilon.
\end{align}
For each $m<M$, Lemma~\ref{lem:strong-cont-Y} gives
$w_m\norm{\tau_hf-f}_{Y_m}\to0$ as $h\to0$.
Since there are only finitely many such indices,
there exists $\rho>0$ such that
\begin{align}\label{1.2}
 \max_{1\le m<M}w_m\norm{\tau_hf-f}_{Y_m}
 <\varepsilon
 \qquad\text{whenever }|h|<\rho.
\end{align}
Combining  estimates \eqref{1.1} and \eqref{1.2}, we obtain
\[
 \norm{\tau_hf-f}_X
 \stackrel{\eqref{eq:X}
 }{=} \sup_{m\ge1}w_m\norm{\tau_hf-f}_{Y_m}
 <\varepsilon
 \qquad\text{whenever }|h|<\rho.
\]
Thus, $\norm{\tau_hf-f}_X\to0 $  as $ h\to0.$

Finally, let $g$ be bounded and compactly supported. 
By Proposition \ref{prop:Y},
\[
    g\in Y_m,
    \qquad
    \norm{g}_{Y_m}\le\norm{g}_\infty
    \quad m\ge1.
\]
It follows from $w_m\to0$ that 
\[
 0\le w_m\norm{g}_{Y_m}
 \le w_m\norm{g}_\infty\to 0.
\]
Thus, $g\in X$ by \eqref{eq:X}.
\end{proof}

\begin{remark}\label{rem:fatou}
Recall that a function lattice $E$ has the
Fatou property if $0\le g_j\uparrow g$ a.e. and
$\sup_j\|g_j\|_E<\infty$ imply
$g\in E$ and $\|g_j\|_E\uparrow\|g\|_E$ \cite{MR928802}. 
The space $(X, \norm{\cdot}_{X})$ (see  \eqref{eq:X}) 
does not have the Fatou property. 
Indeed, let $g_j=\one_{[-j,j]}$.  By
Proposition~\ref{prop:X}, we have $g_j\in X$ and
$\norm{g_j}_X\le\norm{g_j}_\infty=1$ for every $j\ge1$. 
On the other hand, $g_j\uparrow\one_{\R}$ pointwise.
Since 
for each fixed $m$, 
\[
 V_{\one_{\R}}(t)(r,\ell)
\stackrel{\eqref{eq:fiberization}}{=}\one_{\R}\bigl(t+\delta_m(r+N_m\ell)\bigr)
 =1
\]
for every $t\in[0,\delta_m)$, $r\in\G_{N_m}$, and
$\ell\in\Z$, it follows that  $V_{\one_{\R}}(t)$ does not belong to
$c_0(\G_{N_m}\times\Z)$.  
Consequently, $$\one_{\R}\notin Y_m
\stackrel{\text{Prop.\ref{prop:Y}}}{=}
\left\{
f\in L^0(\R):
V_f(t)\in c_0(\G_{N_m}\times\Z)
\text{ for a.e. }t\in[0,\delta_m),
~~\|f\|_{Y_m}<\infty
\right\}$$
for any $m$, and
therefore $\one_{\R}\notin X$ (see \eqref{eq:X}). 
Thus, $X$ does not have  the Fatou property.
\end{remark}



\begin{proof}[Proof of Theorem \ref{thm:main}] 
Proposition \ref{prop:X} establishes assertions (1) and (2). It remains to construct the sequence. 
For $n\ge1$, define 
\begin{equation*}
        f_n:=w_n^{-1}\psi e_{A_n},
\end{equation*} 
where $e_{A_n}$ is  the periodic function defined by
\eqref{eq:pattern} with parameters
$A_n$, $N_n$, and $\delta_n$. 
Every $f_n$ is  bounded and  compactly supported,  and hence belongs to $X$ 
by Proposition \ref{prop:X}.  
Proposition~\ref{prop:critical}  gives 
\begin{equation}\label{eq:own}
                  \frac1{16}\le w_n\|f_n\|_{Y_n}\le1.
\end{equation}

Now, we estimate $\bigl|\{x\in Q:e_{A_n}(x)=1\}\bigr|$. 
For every $j\in\Z$, the definition of $e_{A_n}$ yields
\[
\begin{aligned}
 &\quad\,\,~  \bigl|
 \{x\in[j N_n \delta_n,(j+1) N_n \delta_n ):e_{A_n}(x)=1\}\bigr|
\\  & \,\,= \left|
  \bigsqcup_{k=jN_n}^{(j+1)N_n-1}
    \{x\in[k\delta_n,(k+1)\delta_n):e_{A_n}(x)=1\} 
   \right| 
 \\ &\stackrel{\eqref{eq:pattern} }{=}
 \delta_n
   \sum_{k=jN_n}^{(j+1)N_n-1}\one_{A_n}(k\bmod N_n) =\delta_n
   \sum_{r=0}^{N_n-1}\one_{A_n}(r)=|A_n|\delta_n.
\end{aligned}
\]
Set
\[
 M:=\left| 
 \{j\in\Z:[j N_n \delta_n,(j+1) N_n \delta_n )\subseteq Q\}
 \right| .
\]
 Then $M N_n \delta_n\le|Q|$.
Apart from endpoints, the remaining portion of $Q$
consists of at most two intervals, each of length
at most $N_n \delta_n$.
Consequently,
\begin{align}\label{eq:support-measure}
& \nonumber \quad  \bigl|\{x\in Q:e_{A_n}(x)=1\}\bigr|
\\  & \nonumber 
\,\, \le
\sum_{\substack{j\in\Z\\
[jN_n\delta_n,(j+1)N_n\delta_n)\subseteq Q}}
\bigl|
\{x\in[jN_n\delta_n,(j+1)N_n\delta_n):
e_{A_n}(x)=1\}
\bigr|
+2N_n\delta_n
\\ &\nonumber  \,\, \le M|A_n|\delta_n+2N_n \delta_n
=\frac{|A_n|}{N_n}M N_n \delta_n+2 N_n \delta_n
\le\frac{|A_n|}{N_n}|Q|+2 N_n \delta_n\\
 &
\stackrel{\eqref{fracAnNnlefrac2Kn}}{\le}
 \frac{2|Q|}{K_n}+2 N_n \delta_n
 \stackrel{\eqref{eq:K-choice}, \eqref{choosePn}}{<}
 \frac{d_n}{2}+\frac{d_n}{2}
 =d_n.
\end{align}

Since $0\le\psi\le1$, it follows that
\begin{equation}\label{1.3}
    |f_n|
    \le
    w_n^{-1}\one_{\{x\in Q:e_{A_n}(x)=1\}}.
\end{equation}
If $m<n$, 
using the ideal property and homogeneity of $Y_m$, we obtain
\begin{equation}\label{eq:past}
\begin{aligned}
 w_m\norm{f_n}_{Y_m}
 \stackrel{\eqref{1.3} }{\le}\frac{w_m}{w_n}
 \norm{\one_{\{e_{A_n}=1\}\cap Q}}_{Y_m}\stackrel{\eqref{eq:support-measure}, \eqref{eq:past-small} }{\le}
 \frac{w_m}{w_n}
 \left(2^{-n}\frac{w_n}{w_m}\right)
 =2^{-n}.
\end{aligned}
\end{equation}
We next consider $m>n$.
Since $|\psi e_{A_n}|\le1$ and $\psi e_{A_n}$ is compactly
supported for every $n\ge1$, Proposition~\ref{prop:Y} gives
$\norm{\psi e_{A_n}}_{Y_m}\le1$ for every $m\ge1$.
By homogeneity and the monotonicity of $(w_m)$,
\begin{equation}\label{eq:future}
 w_m\norm{f_n}_{Y_m}=
\frac{w_m}{w_n} \norm{\psi e_{A_n}}_{Y_m} 
 \le\frac{w_m}{w_n}
 \le 
 \frac{w_{n+1}}{w_n}   
\stackrel{\eqref{eq:weight-choice}}{\le}
 \frac{2^{-n}w_n}{w_n}=
2^{-n}, \qquad 
\forall m>n.
\end{equation}
Hence, we conclude that 
\begin{align}\label{cases}
    \begin{cases}
w_m\|f_n\|_{Y_m}
\stackrel{\eqref{eq:past}, \eqref{eq:future}}{\le}
2^{-n},&
m\neq n,
\\
\dfrac1{16} 
\stackrel{\eqref{eq:own}}{\le} w_n\|f_n\|_{Y_n}
\stackrel{\eqref{eq:own}}{\le} 1,&m=n.
\end{cases}
\end{align}
Thus, $ 
    \frac1{16}\le 
 w_n\|f_n\|_{Y_n}
\le  \
\mathop{\sup}\limits_{m\ge1} w_m \norm{f_n}_{Y_m}
 \le1,$ 
i.e., \begin{equation}\label{eq:norm-bounds} 
 \frac1{16}\le  \|f_n\|_X  \leq 1 \quad   (\text{by }~ \eqref{eq:X}).
\end{equation}

We next prove uniform translation continuity of $(f_n)$.
For $m\ne n$, the $p$-triangle inequality and translation
isometry give 
\[
\begin{aligned}
\big(
w_m \norm{\tau_hf_n-f_n}_{Y_m}
\big)^p 
\le  
 w_m^p 
\bigl(\norm{\tau_hf_n}_{Y_m}^p+\norm{f_n}_{Y_m}^p\bigr)
\stackrel{\text{Prop.\ref{prop:Y}}}{=}2w_m^p\norm{f_n}_{Y_m}^p
 \stackrel{\eqref{cases} }{\leq} 2(2^{-n})^p.
\end{aligned}
\]
Hence,
\begin{equation}\label{eq:cross-modulus}
 w_m\norm{\tau_hf_n-f_n}_{Y_m}
 \le2^{1/p}2^{-n},
 \qquad m\ne n.
\end{equation}
For $m=n$, \eqref{eq:critical-modulus} in  Proposition \ref{prop:critical}  yields
\begin{equation}\label{eq:own-modulus}
 w_n\norm{\tau_hf_n-f_n}_{Y_n}
 \le 2^{1/p} (L|h|+K_n^{-1})
\stackrel{\eqref{eq:K-choice}}{\le}
2^{1/p} (L|h|+2^{-n}). 
\end{equation}
Taking the supremum over $m$, we obtain,  for every
$n\ge1$ and $h\in\R$,
\begin{align}\label{1.5}
 \norm{\tau_hf_n-f_n}_X
\stackrel{\eqref{eq:X},  \eqref{eq:cross-modulus}, \eqref{eq:own-modulus}}{ \le}
\max\bigl\{2^{1/p}2^{-n},
               2^{1/p} (L|h|+2^{-n})\bigr\}
           = 2^{1/p} (L|h|+2^{-n}).
\end{align}
Let $\varepsilon>0$. Choose an integer $n_0\ge2$
and $\rho_0>0$ such that
\begin{align}\label{2novare}
    2^{1/p}2^{-n_0}<\frac{\varepsilon}{2},
    \qquad  2^{1/p}L\rho_0<\frac{\varepsilon}{2}.
\end{align}
If  $n\ge n_0$ and $|h|<\rho_0$, then 
\begin{align}\label{1.6}
 \sup_{n\ge n_0}\norm{\tau_hf_n-f_n}_X
\stackrel{\eqref{1.5}, \eqref{2novare}}{<}
\varepsilon.
\end{align}
For each $1\le n<n_0$, Proposition~\ref{prop:X} gives
$\norm{\tau_hf_n-f_n}_X\to0$ as $h\to0$.
There exists $\rho_1>0$ such that
\begin{align}\label{1.7}
 \max_{1\le n<n_0}\norm{\tau_hf_n-f_n}_X<\varepsilon
 \qquad\text{whenever }|h|<\rho_1.
\end{align}
Combining estimates \eqref{1.6} and \eqref{1.7}, we have
$\mathop{\sup}\limits_{n\ge1}\norm{\tau_hf_n-f_n}_X<\varepsilon$
whenever $|h|<\min\{\rho_0,\rho_1\}$.
Therefore,
\begin{equation*}
 \lim_{h\to0}\sup_{n\ge1}
 \norm{\tau_hf_n-f_n}_X=0.
\end{equation*}

 Finally, we show that $(f_n)$ has no convergent
subsequence in $X$.
Taking $m=1$ in \eqref{eq:past} and using $w_1=1$, we obtain, for $n\ge2$,
\[
    \norm{f_n}_{Y_1}
    \le 2^{-n}\to0. 
\] 
Thus $f_n\to0$ in $Y_1$.
Suppose that a subsequence $(f_{n_j})$ converges
to some $f\in X$. By \eqref{eq:X},
\[
    \norm{f_{n_j}-f}_{Y_1}
    \le w_1^{-1}\norm{f_{n_j}-f}_X
    \to0.
\]
Uniqueness of the limit in $Y_1$ (see Lemma~\ref{lem:Y-local-measure}) therefore gives
$f=0$ a.e.
Consequently,
\[
 \norm{f_{n_j}}_X
 =\norm{f_{n_j}-f}_X\to 0,
\]
contradicting the lower bound
$\norm{f_{n_j}}_X\ge1/16$ in \eqref{eq:norm-bounds}.
Thus, $(f_n)$ has no convergent subsequence in $X$. 

By the preliminary observations, \((X,\norm{ \cdot}_X^p)\in\mathcal L\); common compact support gives uniform vanishing at infinity, so the sequence disproves both conjectures.
\end{proof}

\begin{remark}\label{notabsonorm}
The sequence $(f_n)$ is not norm absolutely
equicontinuous. 
Set
\begin{align}\label{defFj}
 F_j:=\bigcup_{n\ge j}\{x\in\R:f_n(x)\ne0\}.
\end{align}
Then,  $F_{j+1}\subseteq F_j\subseteq Q$ 
and 
\begin{align*}
    |F_j|
    &\quad \le
    \sum_{n\ge j}|\{x\in\R:f_n(x)\ne0\}|
  \stackrel{\eqref{1.3}}{  \leq } 
   \sum_{n\ge j}
   |\{x\in Q:e_{A_n}(x)=1\}|
  \\ &  \stackrel{\eqref{eq:support-measure}, \eqref{defdnrange}}{  \leq } 
    \sum_{n\ge j}
    2^{-n}
    =2^{1-j}\to0 \qquad\text{as }j\to\infty.
\end{align*}
Since $\bigcap_{k\ge1}F_k\subseteq F_j$ for every $j\geq 1$,
$ 0\le\left|\bigcap_{k\ge1}F_k\right|
 \le|F_j|\to 0.
$ 
Thus, $\bigcap_{k\ge1}F_k$ has measure zero.
Moreover,
$ 
 \one_{F_j}(x)\downarrow
 \one_{\bigcap_{k\ge1}F_k}(x)
$  for $x\in\R$, 
i.e., 
$\one_{F_j}\downarrow0$ almost everywhere. 
For every $n\ge j$, 
$\{x:f_n(x)\ne0
\}
\stackrel{\eqref{defFj}}{\subseteq}  F_j$.
Thus, $f_n$ vanishes outside $F_j$, and consequently
$ 
 f_n\one_{F_j}=f_n.
$ 
Hence,
\[
\begin{aligned} \sup_{n\ge1}\norm{f_n\one_{F_j}}_X
 \ge\sup_{n\ge j}\norm{f_n\one_{F_j}}_X =\sup_{n\ge j}\norm{f_n}_X
  \ge
 \norm{f_j}_X  \stackrel{\text{Thm.\ref{thm:main}}
 }{\ge} 
 \frac1{16}.
\end{aligned}
\]
Thus, the family is not norm absolutely equicontinuous.
This does not contradict order continuity of $X$:
for each fixed $n$, one has
$\norm{f_n\one_{F_j}}_X\to0$ as $j\to\infty$,
but the convergence is not uniform in $n$.

\end{remark}

\section*{Acknowledgments}
The authors acknowledge the assistance of ChatGPT 6 Astra. Suggestions obtained through interactions with this AI system were subsequently examined, reformulated, incorporated into the manuscript, and
independently verified by the authors. The authors take full responsibility
for all mathematical content in the final manuscript.


\begin{thebibliography}{99}


\bibitem{Arzela1895}
C. Arzel\`a,
\textit{Sulle funzioni di linee}, Mem. Accad. Sci. Ist. Bologna Cl. Sci. Fis. Mat.
 \textbf{5}(5) (1895), 55--74.

\bibitem{Ascoli1884}
G. Ascoli,
\textit{Le curve limite di una variet\`a data di curve},
Atti della R. Accad. Dei Lincei Memorie della Cl. Sci. Fis. Mat. Nat. (III) 
 \textbf{18} (1884), 521--586.


\bibitem{MR928802}
C. Bennett, R. Sharpley,
\textit{Interpolation of Operators},
Pure and Applied Mathematics, vol.~129,
Academic Press, Inc., Boston, MA, 1988.

\bibitem{Brudnyi2015}
Yu. Brudnyi,
\textit{Compactness criteria for spaces of measurable functions},
St. Petersburg Math. J. \textbf{26}(1) (2015), 49--68.
Republication of the English-language original in
Algebra i Analiz \textbf{26}(1) (2014),  68--93.


\bibitem{CaetanoGogatishviliOpic2016}
A. Caetano, A. Gogatishvili, B. Opic,
\textit{Compactness in quasi-Banach function spaces and
applications to compact embeddings of Besov-type spaces},
Proc. Roy. Soc. Edinburgh Sect. A
\textbf{146}(5) (2016), 905--927.


\bibitem{ConejeroMullerPeris2007}
J.~A. Conejero, V. M{\"u}ller, A. Peris,
\textit{Hypercyclic behaviour of operators in a
hypercyclic $C_0$-semigroup},
J. Funct. Anal. \textbf{244}(1) (2007), 342--348.




\bibitem{Frechet1937}M. Fr\'echet,\textit{Sur les ensembles compacts de fonctions de carr\'es sommables},Acta Sci. Math. (Szeged) \textbf{8} (1937), 116--126.

\bibitem{Frink1937}
A.~H. Frink,
\textit{Distance functions and the metrization problem},
Bull. Amer. Math. Soc. \textbf{43}(2) (1937),  133--142.





\bibitem{GuoZhao2020}
W. Guo, G. Zhao,
\textit{On relatively compact sets in quasi-Banach function spaces},
Proc. Amer. Math. Soc. \textbf{148}(8) (2020), 3359--3373.


\bibitem{HancheOlsenHolden2010}
H. Hanche-Olsen, H. Holden,
\textit{The Kolmogorov--Riesz compactness theorem},
Expo. Math. \textbf{28}(4) (2010), 385--394.


\bibitem{HuangNessipbayevSukochevZanin2025}
J. Huang, Y. Nessipbayev, F. Sukochev, D. Zanin,
\textit{Compactness criteria in quasi-Banach symmetric operator
spaces associated with a non-commutative torus},
J. Funct. Anal. \textbf{289}(5) (2025), 
Paper No.~110946, 35~pp.


\bibitem{Kalton1980}
N.~J. Kalton,
\textit{Linear operators on $L_p$ for $0<p<1$},
Trans. Amer. Math. Soc. \textbf{259}(2) (1980),  319--355.


\bibitem{Kolmogoroff1931}
A.~N. Kolmogoroff,
\textit{{\"U}ber die Kompaktheit der Funktionenmengen
bei der Konvergenz im Mittel},
Nachr. Ges. Wiss. G\"ottingen Math.-Phys. Kl.
\textbf{1931} (1931), 60--63.

\bibitem{Krotov2012}
V.~G. Krotov,
\textit{Criteria for compactness in $L^p$-spaces, $p\geq0$},
Sb. Math. \textbf{203}(7) (2012), 1045--1064.





\bibitem{Pisier2016}
G. Pisier,
\textit{Martingales in Banach Spaces},
Cambridge Studies in Advanced Mathematics, vol.~155,
Cambridge University Press, Cambridge, 2016.


\bibitem{Riesz1933}
M. Riesz,
\textit{Sur les ensembles compacts de fonctions sommables},
Acta Sci. Math. (Szeged) \textbf{6} (1933), 136--142.

\bibitem{Rolewicz1972}
S. Rolewicz,
\textit{Metric Linear Spaces},
Monografie Matematyczne, vol.~56,
PWN--Polish Scientific Publishers, Warsaw, 1972.







\end{thebibliography}
\end{document}